\documentclass[11pt]{article}

\usepackage[letterpaper,margin=1in]{geometry}
\usepackage[T1]{fontenc}
\usepackage{amsmath,amssymb,amsthm,mathtools}
\usepackage{microtype}
\usepackage{booktabs,tabularx}

\usepackage[
  backend=biber,
  style=alphabetic,
  maxbibnames=99
]{biblatex}
\usepackage[
  colorlinks=true,
  citecolor=blue,
  linkcolor=blue,
  urlcolor=blue
]{hyperref}
\usepackage{booktabs}
\usepackage{tabularx}
\usepackage{array}
\usepackage{ragged2e}
\newcolumntype{Y}{>{\RaggedRight\arraybackslash}X}
\newcolumntype{P}[1]{>{\RaggedRight\arraybackslash}p{#1}}

\newtheorem{theorem}{Theorem}[section]
\newtheorem{lemma}[theorem]{Lemma}
\newtheorem{proposition}[theorem]{Proposition}
\newtheorem{corollary}[theorem]{Corollary}

\theoremstyle{remark}
\newtheorem{remark}[theorem]{Remark}

\theoremstyle{definition}

\newcommand{\R}{\mathbb{R}}
\newcommand{\F}{\mathcal{F}}

\newcommand{\eps}{\varepsilon}
\newcommand{\HD}{\operatorname{HD}}

\makeatletter
\renewcommand{\@fnsymbol}[1]{%
  \ifcase#1
  \or *
  \or **
  \or ***
  \else
    \@ctrerr
  \fi
}
\makeatother

\title{New Quantitative Bounds for the $(p,q)$-Theorem for Unions of Convex Sets}

\author{
  Chaya Keller%
  \thanks{School of Computer Science, Ariel University, Israel.
  \texttt{chayak@ariel.ac.il}.}
  \and
  Shakhar Smorodinsky%
  \thanks{Institute for the Theory of Computing,
  Faculty of Computer and Information Science,
  Ben-Gurion University of the Negev,
  Be'er Sheva 84105, Israel.
  \texttt{shakhar@bgu.ac.il}.}
}

\date{}

\begin{document}
\maketitle

\begin{abstract}
A set in $\R^d$ is called $s$-convex if it is the union of at most $s$ convex sets. 
A family of sets $\F$ is said to satisfy the $(p,q)$ property if among any $p$ sets in $\F$, some
$q$ have a non-empty intersection. Let $\HD_d^{(s)}(p,q)$ denote the minimum number of points needed to pierce a finite family of $s$-convex sets that satisfies the $(p,q)$-property. Alon and Kalai (1995) proved that $\HD_d^{(s)}(p,q)$ exists for any $p \geq q \geq d+1$ and any $s \geq 1$, but the quantitative bounds they obtained are very loose.

We present several improved upper and lower bounds, for a general $d$ and in the well-studied setting of families of $s$-intervals of the line (i.e., $\HD_1^{(s)}(p,q)$). In particular, we prove the following:

\noindent (i) For every $d\ge2$, $s \geq 1$ and $\delta>0$, if $p>q$ and $q\ge C_d\log(e sp)$, then
$\HD_d^{(s)}(p,q)
\le
p-q+1
+
O_{d,\delta}((s \cdot \tfrac{p}{q} \cdot \log \tfrac{esp}{q})^{\rho_d+\delta}),$
where $\rho_d<d$ is the exponent in the weak epsilon-net theorem of Rubin (2022). 

\noindent (ii) For $s \geq 1$, $p \geq q \geq 2$ and $q\ge C_0s\log(2s)\log(ep)$, $p-q+s \leq \HD_1^{(s)}(p,q) \leq p-q+2s+1$. This result, whose proof uses the bootstrapping technique of Keller, Smorodinsky and Tardos (2018), the improved bound of Zerbib (2019) on $\HD_1^{(s)}(p,p)$ and geometric methods, provides the first near-tight estimate for $\HD_1^{(s)}(p,q)$ for $q>2$. 

\noindent (iii) For any fixed $s$, there are an integer $\kappa_s\in\{s,\ldots,2s\}$ and constants $C_s,p_s>0$ such that, whenever $p\ge p_s$ and $q\ge C_s\log(ep)$,
$ 
\HD_1^{(s)}(p,q)\in\{p-q+\kappa_s,\;p-q+\kappa_s+1\}.
$
\noindent Interestingly, this two-value concentration result holds, although the exact value of the threshold remains unknown.

\noindent (iv) For any $s \geq 1$, $\HD_3^{(s)}(p,4) \geq sp^{2-o(1)}$. 
Already in the classical setting of families of convex sets, this significantly improves the best known lower bound on $\HD_d^{(1)}(p,d+1)$, for all $d \geq 3$.
\end{abstract}

\section{Introduction}
\subsection{Background}

\paragraph{Helly's theorem and the $(p,q)$ problem.}
Helly's theorem states that if $\mathcal F$ is a finite family of convex sets in $\R^d$ and every $d+1$ members of $\mathcal F$ have a nonempty intersection, then $\bigcap_{F\in\mathcal F}F\neq\emptyset$. More generally, the Helly number of a class of sets is the smallest integer $h$, if it exists, such that every finite family from the class, whose subfamilies of size at most $h$ have nonempty intersection, has a nonempty total intersection. Thus, the Helly number of convex sets in $\R^d$ is $d+1$.

Let $p\ge q$ be positive integers. A family $\mathcal F$ with at least $p$ members satisfies the $(p,q)$-property if among every $p$ members of $\mathcal F$, some $q$ have a nonempty intersection. A transversal, or a piercing set, for $\mathcal F$ is a set of points meeting every member of $\mathcal F$, and $\tau(\mathcal F)$, the transversal number of $\mathcal F$, denotes the minimum size of such a set.

Hadwiger and Debrunner~\cite{HadwigerDebrunner57} asked whether, for every $p\ge q\ge d+1$, the $(p,q)$-property forces a bounded transversal for families of compact convex sets in $\R^d$. In their celebrated $(p,q)$-theorem, Alon and Kleitman~\cite{AlonKleitman92} proved that it does: for every such $d,p,q$, there is an integer depending only on $d,p,q$ that bounds the minimal size of transversal for every family of compact\footnote{The theorem is often stated for compact convex sets. For finite families, compactness is inessential.} convex sets satisfying the $(p,q)$-property. We denote the least such integer by $\HD_d(p,q)$ and call it the Hadwiger--Debrunner number. A simple construction gives the universal lower bound $\HD_d(p,q)\ge p-q+1$. Hadwiger and Debrunner~\cite{HadwigerDebrunner57} proved that this lower bound is attained whenever $q>\frac{d-1}{d}p+1$, namely, $\HD_d(p,q)=p-q+1$ throughout this range. Determining the quantitative behavior of these numbers outside this range remains a central problem in discrete geometry; see~\cite{KST18,KS21,RocheNewton26} for the best known upper and lower bounds.


\paragraph{Unions of convex sets.}
Unions of a bounded number of convex sets form one of the basic extensions of convexity. A set in $\R^d$ is called $s$-convex if it is the union of at most $s$ convex sets. For $s>1$, Larman~\cite{Larman68} showed that the class of all $s$-convex sets has no finite Helly number in general. This failure motivated a substantial line of work identifying additional assumptions under which Helly-type conclusions can be obtained. In particular, positive results are known when one controls the structure or topology of all intersections; see the results of Amenta~\cite{Amenta96}, Matou\v{s}ek~\cite{Matousek97}, and Kalai and Meshulam~\cite{KalaiMeshulam08}.

Despite the absence of an ordinary Helly theorem, strong piercing results remain possible. Alon and Kalai~\cite{AlonKalai95} proved a qualitative $(p,q)$-theorem for unions of convex sets: for fixed $d,s,p,q$, with $p\ge q\ge d+1$, every finite family of unions of at most $s$ compact convex sets satisfying the $(p,q)$-property has a transversal of bounded size. We write $\HD_d^{(s)}(p,q)$ for the least integer that pierces every finite family of nonempty $s$-convex sets in $\R^d$ satisfying the $(p,q)$-property.

Unions of convex sets have recently also been studied from the Radon and Tverberg point of view, both in the context of abstract separable spaces \cite{AlonSmorodinsky26} and in the Euclidean setting \cite{AlonSmorodinsky26,GeShuXu26}.


\paragraph{Unions of convex sets in dimension 1.}
The one-dimensional case, where $s$-convex sets are unions of at most $s$ intervals, has a rich theory of its own. Kaiser~\cite{Kaiser97} proved that every family of $s$-intervals satisfying the $(p,2)$-property can be pierced by at most $(s^2-s+1)(p-1)$ points. Matou\v{s}ek~\cite{Matousek01} constructed separated $s$-interval families satisfying the $(p,2)$-property for which at least $c s^2(p-1)/(\log s)^2$ piercing points are needed. Thus the quadratic dependence on $s$ was known up to a polylogarithmic factor. Alon~\cite{Alon98} later gave a short elementary proof of the comparable upper bound $2s^2(p-1)$.

For $q>2$, no comparably sharp picture was known. Zerbib~\cite{Zerbib19} proved the general upper bound $O_{p,q}(s^{q/(q-1)})$, where the dependence on $p$ is polynomial. Her paper also contains an elegant topological theorem, based on Komiya's theorem, for separated $s$-intervals with the $(p,p)$-property, but no matching lower bound was known.

\paragraph{Fractional Helly theorem and weak $\eps$-nets.}
The proof of the Alon--Kleitman theorem \cite{AlonKleitman92}, as well as the extension of Alon and Kalai \cite{AlonKalai95}, is built around two geometric ingredients: the fractional Helly theorem and the weak $\eps$-net theorem.

The fractional Helly theorem of Katchalski and Liu~\cite{KatchalskiLiu79} is a density version of Helly's theorem. It states that for every $d$ and every $\alpha>0$, there is $\beta_d(\alpha)>0$ such that if at least $\alpha\binom{n}{d+1}$ of the $(d+1)$-tuples in a family of $n$ convex sets in $\R^d$ have nonempty intersection, then some point belongs to at least $\beta_d(\alpha)n$ members of the family. The elementary lexicographic proof already gives the useful linear estimate $\beta_d(\alpha)\ge \alpha/(d+1)$.

Let $P$ be a finite multiset of points in $\R^d$. A weak $\eps$-net for $P$ is a set $N\subseteq\R^d$, such that $N$ meets every convex set containing at least $\eps|P|$ points of $P$. We denote by $f_d(\eps)$ the smallest size that can always be guaranteed. Alon, B\'ar\'any, F\"uredi and Kleitman~\cite{AlonBFK92} proved that $f_d(\eps)$ is finite for every fixed $d$ and $\eps>0$. The classical general estimate is $f_d(\eps)=\widetilde O_d(\eps^{-d})$ for $d\ge2$; see~\cite{AlonBFK92,ChazelleEtAl95,MatousekWagner04}. Rubin~\cite{Rubin21,Rubin22} improved the exponent in every dimension $d\ge2$. Define
\[
\rho_d=
\begin{cases}
3/2, & d=2,\\
2.558, & d=3,\\
3.48, & d=4,\\
\bigl(d+\sqrt{d^2-2d}\bigr)/2, & d\ge5.
\end{cases}
\]
Then, for every $\delta>0$, one has $f_d(\eps)=O_{d,\delta}(\eps^{-(\rho_d+\delta)})$. We use these estimates throughout the paper. In dimension one, intervals admit weak $\eps$-nets of size $O(1/\eps)$, which is obviously optimal.

The Alon--Kleitman argument combines these two results as follows. A double-counting argument first produces many intersecting $(d+1)$-tuples. Fractional Helly then gives a point contained in a positive fraction of the sets. Linear-programming duality produces a finite multiset of points such that every member of the original family contains a fixed positive fraction of them. A weak $\eps$-net for this multiset is then a transversal for the original family. 

Keller, Smorodinsky and Tardos~\cite{KST18} sharpened the quantitative part of this scheme by counting intersecting $q$-tuples more efficiently and by applying a dichotomy method. In particular, for every fixed $d$ and every $\eps>0$, they proved that, for all sufficiently large $p$ and every $q\ge p^{(d-1)/d+\eps}$, one has $p-q+1\le \HD_d(p,q)\le p-q+2$. 

\subsection{Our results}

We obtain improved upper and lower bounds on the numbers $\HD_d^{(s)}(p,q)$ in several regimes, using a large variety of geometric and combinatorial techniques.

\subsubsection{Upper bounds on the $(p,q)$-theorem for a general dimension $d$}

Our first result is a general quantitative $(p,q)$-theorem for families of $s$-convex sets, which improves over direct reduction to all convex components by essentially a factor $s^{\rho_d}$.

\begin{theorem}\label{thm:q=d+1}
For every fixed $d\ge2$, every $\delta>0$, and all $s\ge1$ and $p\ge d+1$, one has $$\HD_d^{(s)}(p,d+1)=O_{d,\delta}((s^{d+1}p^d)^{\rho_d+\delta}).$$
\end{theorem}

For larger values of $q$, we prove a significantly stronger bound using the $(p,q)$ dichotomy method of~\cite{KST18}. Like in Theorem~\ref{thm:q=d+1}, the `error term' added to $p-q+1$ is smaller by essentially a factor $s^{\rho_d}$ than the error term obtained from  direct reduction to all convex components and application of the dichotomy method. 

\begin{theorem}\label{thm:large-q}
For every fixed $d\ge2$ and $\delta>0$, there is a constant $C_d>0$ such that, if $p>q\ge d+1$ and $q\ge C_d\log(e sp)$, then
\[
\HD_d^{(s)}(p,q)\le p-q+1+O_{d,\delta}\!\left(\left(s \cdot \frac{p}{q} \cdot \log \left(\frac{esp}{q}\right)\right)^{\rho_d+\delta}\right).
\]
\end{theorem}

As a consequence from the proof method of the theorems above, we obtain the following fractional Helly estimate, which is of independent interest. Its dependence on $s$ is one power better than the direct reduction to convex components.

\begin{theorem}[Fractional Helly for $s$-convex sets]\label{thm:fractional-helly}
Let $F_1,\ldots,F_n\subseteq\R^d$ be $s$-convex sets, where $n\ge d+1$. If at least $\alpha\binom{n}{d+1}$ of their $(d+1)$-tuples intersect, then some point belongs to at least $$\frac{\alpha n}{(d+1)s^d}$$ of the sets.
\end{theorem}

A colorful tuple from families $\F_1,\ldots,\F_{d+1}$ consists of one member of each family. For ordinary convex sets, colorful fractional Helly theorems, including optimal bounds, were developed in~\cite{BulavkaGoodarziTancer21}. We also obtain the following colorful version of Theorem~\ref{thm:fractional-helly}.

\begin{theorem}[Colorful fractional Helly for $s$-convex sets]\label{thm:colorful-fh}
Let $\F_1,\ldots,\F_{d+1}$ be finite nonempty families of $s$-convex sets in $\R^d$, with $|\F_i|=n_i$. If at least $\alpha n_1\cdots n_{d+1}$ colorful tuples intersect, then for some $i\in[d+1]$ there is a point contained in at least $\alpha n_i/((d+1)s^d)$ members of $\F_i$.
\end{theorem}

\subsubsection{Upper bounds on the $(p,q)$-theorem for families of $s$-intervals}

When $d=1$, $s$-convex sets are precisely unions of at most $s$ compact intervals, usually called $s$-intervals. For $q=2$, the aforementioned upper bounds of Kaiser~\cite{Kaiser97} and Alon~\cite{Alon98} and the lower bound of Matou\v{s}ek~\cite{Matousek01} determine the quadratic dependence on $s$ up to a factor $(\log s)^2$. 
For $q>2$, Zerbib~\cite{Zerbib19} proved that, for fixed $p\ge q>1$, every family of $s$-intervals with the $(p,q)$-property can be pierced by $O_{p,q}(s^{q/(q-1)})$ points, where the dependence on $p$ is polynomial. No lower bound of comparable order was known. Our first result in this setting recovers the same exponent of $s$, with explicit linear dependence on $p$, and obtains a significantly stronger bound in the range $q \geq C \log(esp)$. The proof is based on combination of the proofs of Theorems~\ref{thm:q=d+1} and~\ref{thm:large-q} with the fact that families of intervals on the line admit weak $\epsilon$-nets of size $O(1/\epsilon)$.   

\begin{theorem}[General bounds for $s$-intervals]\label{thm:one-dimensional-intro}
For all $p\ge q\ge2$ and $s\ge1$, one has $$\HD_1^{(s)}(p,q)=O\!\left(p s^{q/(q-1)}\right).$$ If $p>q$ and $q\ge C\log(e sp)$, then
\[
\HD_1^{(s)}(p,q)\le p-q+1+O\!\left(s \cdot \frac{p}{q} \cdot\log \frac{e sp}{q}\right).
\]
\end{theorem}

Our main one-dimensional theorem recovers $\HD_1^{(s)}(p,q)$ up to an additive term of $s+1$, when $q$ exceeds $s$ by a polylogarithmic factor. Its proof combines our geometric estimates with the $(p,q)$ dichotomy method of \cite{KST18} and an iterated bootstrapping argument, using Zerbib's aforementioned diagonal theorem as the terminal case. 

\begin{theorem}\label{thm:one-dimensional-polylog}
There is an absolute constant $C_0>0$ such that, for all integers $p\ge q\ge2$ and $s\ge2$, if $q\ge C_0s\log(2s)\log(ep)$, then
\[
p-q+s\le \HD_1^{(s)}(p,q)\le p-q+2s+1.
\]
\end{theorem}

For a fixed number of components, the preceding estimate can be sharpened from an interval of $s+2$ possible values to two consecutive values.

\begin{theorem}[Two-value concentration for fixed $s$]\label{thm:one-dimensional-concentration}
For every fixed integer $s\ge2$, there exist an integer $\kappa_s\in\{s,\ldots,2s\}$, a constant $C_s>0$, and an integer $p_s$ such that, for all $p\ge p_s$ and $C_s\log(ep)\le q\le p$,
\[
\HD_1^{(s)}(p,q)\in\{p-q+\kappa_s,\;p-q+\kappa_s+1\}.
\]
Moreover, $\kappa_s$ is the eventual constant value of the nonincreasing sequence $\HD_1^{(s)}(t,t)$.
\end{theorem}

The exact value of $\kappa_s$ is not known to us; in particular, our argument does not determine whether $\kappa_s=s$. Thus, Theorem~\ref{thm:one-dimensional-concentration} gives a strong concentration phenomenon without explicitly identifying the two values.

The comparison of our results for $d=1$ with the previous ones, is demonstrated in the table at the beginning of Section \ref{sec:one-dimensional}.

\subsubsection{Lower bounds on the $(p,q)$-theorem}

Our lower-bound results address separately the settings of $q=d+1$ and of an arbitrarily large $q \leq p$. For $q=d+1$ and $s=1$ (i.e., for families of convex sets), lower bounds for the Hadwiger--Debrunner numbers are much less understood than upper bounds. For a long time, the best known lower bound was $\HD_d(p,q) \geq \Omega(\frac{p}{q}\log^{d-1}(\frac{p}{q}))$, obtained by Bukh, Matou\v{s}ek and Nivasch~\cite{BMN11}, and it wasn't clear at all that $\HD_d(p,d+1)$ is not quasi-linear in $p$. 
In~\cite{KS21}, Keller and Smorodinsky used a novel construction of sets of points in general position in the plane due to Balogh and Solymosi~\cite{BS18} to obtain the essentially first superlinear lower bounds $\HD_2(p,3) \geq \Omega(p^{6/5-\epsilon})$ and  $\HD_2(p,q)\ge p^{1+\Omega(1/q)}$ for all $q \geq 3$, which implies that for each $d \geq 2$ and $p \geq q \geq d+1$, 
$\HD_d(p,q)\ge p^{1+\Omega(1/q)}$. Very recently, Roche-Newton~\cite{RocheNewton26} obtained the improved bound $\HD_2(p,3) \geq \Omega(p^{5/4-\epsilon})$.

Using a higher-dimensional general-position construction of Suk and Zeng~\cite{SukZeng26}, we improve the lower bound in dimensions $d \geq 3$ to an almost quadratic one.

\begin{theorem}[The convex lower bound]\label{thm:convex-lower}
For every fixed $d\ge3$ and every $\gamma>0$, there are arbitrarily large $p$ such that
$\HD_d(p,d+1)\ge c_{d,\gamma}p^{2-\gamma}$.
Equivalently, $\HD_d(p,d+1)\ge p^{2-o(1)}$ along an unbounded sequence of values of $p$.
\end{theorem}

Thus, already for ordinary convex sets, the exponent improves from $1+\Omega(1/q)$ to $2-o(1)$. Simply viewing the same construction as a family of $s$-convex sets would give no further dependence on $s$. Thus we use an amplification lemma to convert any convex lower bound into a lower bound for unions, with an additional linear gain in the number of components. Combining it with Theorem~\ref{thm:convex-lower} gives:
\begin{theorem}[Lower bound for $s$-convex sets]\label{thm:s-lower}
For every fixed $d\ge3$ and every $\gamma>0$, there are arbitrarily large $p$ such that, for every $s\ge1$, one has $\HD_d^{(s)}(p,d+1)\ge c_{d,\gamma}s\,p^{2-\gamma}$. Thus $$\HD_d^{(s)}(p,d+1)\ge s\,p^{2-o(1)}.$$
\end{theorem}

For large $q$, the phenomenon is different. Keller, Smorodinsky and Tardos~\cite{KST18} proved the almost exact upper bound $\HD_d(p,q)\le p-q+2$ for ordinary convex sets in a broad large-$q$ range. For unions of convex sets, the additive term must grow both with the number of components and with the dimension. The following construction, that uses the Steiner convex partition theorem of Dumitrescu, Har-Peled and T\'{o}th~\cite{DumitrescuHarPeledToth14}, generalizes the one-dimensional lower bound $p-q+s$ for $d=1$.

\begin{theorem}[A dimension-dependent lower bound]\label{thm:large-q-lower}
Let $d\ge1$, $s\ge1$, and $p\ge q\ge d+1$. Then
\[
\HD_d^{(s)}(p,q)\ge p-q+d(s-1)+1.
\]
In particular, for $d\ge2$ and $s\ge2$, the convex bound $p-q+2$ does not extend to $s$-convex sets.
\end{theorem}

\paragraph{Organization of the paper.} In Section~\ref{sec:upper-d}
 we present the upper bounds for the $(p,q)$-theorem for a general dimension $d$. The results in dimension $1$ (i.e., $s$-intervals) are presented in Section~\ref{sec:one-dimensional}. The lower bound results are presented in Section~\ref{sec:lower}. The proof of the upper bound in Theorem~\ref{thm:one-dimensional-polylog} is presented in Appendix~\ref{app:one-dimensional-bootstrapping} and the proof of Theorem~\ref{thm:colorful-fh} is presented in Appendix~\ref{app:colorful-fh}.

\section{Upper bounds in dimensions \texorpdfstring{$d\ge2$}{d >= 2}}\label{sec:upper-d}

\subsection{The basic bound and the case \texorpdfstring{$q=d+1$}{q = d + 1}}

The certificate argument below is inspired by the lexicographic arguments of~\cite{KatchalskiLiu79}. We use the lexicographic order on $\R^d$. For points
$x=(x_1,\ldots,x_d)$ and $y=(y_1,\ldots,y_d)$, we write
$x<_{\mathrm{lex}}y$ if, at the first coordinate in which $x$ and $y$
differ, the coordinate of $x$ is smaller. If $K\subseteq\R^d$ is
nonempty and compact, its lexicographic minimum is the unique point
$v\in K$ such that no point $x\in K$ satisfies $x<_{\mathrm{lex}}v$.

We shall use the following standard consequence of Helly's theorem;
see~\cite[Section~8.1]{Matousek02}.

\begin{lemma}[Lexicographic minimum]\label{lem:lex}
Let $m\geq d \geq 1$ and let $K_1,\ldots,K_m\subseteq\R^d$ be compact convex sets with nonempty
intersection, and let $v$ be the lexicographic minimum of
$K_1\cap\cdots\cap K_m$. Then there is a set $I\subseteq[m]$ with
$|I|\le d$ such that $v$ is also the lexicographic minimum of
$\bigcap_{i\in I}K_i$.
\end{lemma}

From this point we follow the proof steps of \cite{AlonKleitman92} mentioned above, and improve the first two steps, inspired by \cite{KST18}. For a finite family $\F$, let $N_q(\F)$ be the number of intersecting $q$-tuples, and let $m(\F)$ be the maximum number of members containing one point. The following certificate count is the main geometric input.

\begin{lemma}\label{lem:count}
Let $\F=\{F_1,\ldots,F_n\}$ be a family of $s$-convex sets in $\R^d$, and let $q\ge d+1$. If $m=m(\F)$, then $N_q(\F)\le s^d\binom nd\binom{m-d}{q-d}$. The same statement holds for labelled multisets, where $m(\F)$ counts members with multiplicity.
\end{lemma}

\begin{proof}
Write every $F_i$ as a union of at most $s$ convex components. We may assume that the components are compact. Indeed, for every intersecting $q$-tuple $I$, choose a point $x_I\in\bigcap_{i\in I}F_i$ and, in each $F_i$ with $i\in I$, one component containing $x_I$. Replace every component by the convex hull of the finitely many points assigned to it. All intersecting $q$-tuples remain intersecting, and $m$ does not increase.

For every intersecting $q$-tuple $I$, let $v_I$ be the lexicographic minimum of its intersection. Choose one component of each member of $I$ containing $v_I$. Lemma~\ref{lem:lex} shows that $d$ of these components already determine $v_I$; if fewer are needed, add arbitrary indices from $I$. Thus $I$ receives one of at most $s^d\binom nd$ certificates.

A fixed certificate determines one point $v$. Every tuple carrying it contains the $d$ recorded sets, and its remaining $q-d$ sets must also contain $v$. At most $m-d$ further members contain $v$, so the certificate extends to at most $\binom{m-d}{q-d}$ tuples.
For a labelled multiset, regard all labelled copies as distinct indexed members and count $m$ with multiplicity. The argument above then applies verbatim.
\end{proof}

As an immediate consequence of the certificate count, we obtain Theorem~\ref{thm:fractional-helly}, which is not used in the proofs of the $(p,q)$ results but is of independent interest.

\begin{proof}[Proof of Theorem~\ref{thm:fractional-helly}]
Put $\F=\{F_1,\ldots,F_n\}$ and let $m=m(\F)$. By the hypothesis and Lemma~\ref{lem:count}, applied with $q=d+1$, we have $\alpha\binom n{d+1}\le N_{d+1}(\F)\le s^d\binom nd(m-d)$. Hence $m\ge d+\alpha(n-d)/((d+1)s^d)\ge\alpha n/((d+1)s^d)$, as required.
\end{proof}

\begin{remark}\label{rem:combined-fh}
If $\beta_d$ is any fractional Helly function for ordinary convex sets in $\R^d$, the direct component reduction and Theorem~\ref{thm:fractional-helly} together give $\beta_{d,s}(\alpha)\ge\max\{\beta_d(\alpha/s^{d+1}),\alpha/((d+1)s^d)\}$. The first term can be better for small $s$; the second has the better order in $s$.
\end{remark}

Theorem~\ref{thm:colorful-fh} gives a colorful analogue of Theorem~\ref{thm:fractional-helly}. It is not used in the $(p,q)$ arguments; its proof is deferred to Appendix~\ref{app:colorful-fh}.

\medskip 

 We use the following Tur\'an-type theorem of de Caen~\cite{deCaen83}: If $\mathcal H$ is a $q$-uniform hypergraph on $n$ vertices with no independent set of size $p$, where $n\ge p\ge q$, then
$|E(\mathcal H)|\ge \frac{n-p+1}{n-q+1}\frac{\binom{n}{q}}{\binom{p-1}{q-1}}$.
Here, an independent set is a set of vertices containing no hyperedge. Combining this result with Lemma~\ref{lem:count} gives the following depth estimate.

\begin{lemma}[A deep point]\label{lem:depth}
Let $p\ge q\ge d+1$, and let $\F$ be a family of $n\ge2p$ $s$-convex sets satisfying the $(p,q)$-property. Then some point belongs to at least $c_d qn/(s^{d/(q-d)}p^{(q-1)/(q-d)})$ members of $\F$.
\end{lemma}

\begin{proof}
Let $\mathcal H$ be the $q$-uniform hypergraph whose vertex set is $[n]$, and in which a set $I\in\binom{[n]}{q}$ is a hyperedge if and only if $\bigcap_{i\in I}F_i\ne\emptyset$. Thus $|E(\mathcal H)|=N_q(\F)$. Since $\F$ satisfies the $(p,q)$-property, every set of $p$ vertices of $\mathcal H$ contains a hyperedge. Equivalently, $\mathcal H$ has no independent set of size $p$.

De Caen's theorem therefore gives
$N_q(\F)\ge \frac{n-p+1}{n-q+1}\frac{\binom{n}{q}}{\binom{p-1}{q-1}}$.
Since $n\ge2p$, we have $(n-p+1)/(n-q+1)\ge1/2$. Moreover,
$\frac{\binom{n}{q}}{\binom{p-1}{q-1}}=\frac{n}{q}\prod_{j=1}^{q-1}\frac{n-j}{p-j}\ge\frac{n}{q}\left(\frac{n}{p}\right)^{q-1}$,
where the inequality follows from $n\ge p$. Consequently,
$N_q(\F)\ge n^q/(2qp^{q-1})$.

Let $m=m(\F)$ be the maximum number of members of $\F$ containing one point. By Lemma~\ref{lem:count},
$$N_q(\F)\le s^d\binom{n}{d}\binom{m-d}{q-d}.$$
Using $\binom{n}{d}\le n^d/d!$ and $\binom{m-d}{q-d}\le m^{q-d}/(q-d)!$, we obtain
$N_q(\F)\le s^dn^dm^{q-d}/(d!(q-d)!)$.
Comparing the lower and upper bounds for $N_q(\F)$ implies
$$m^{q-d}\ge\frac{d!(q-d)!}{2q}\frac{n^{q-d}}{s^dp^{q-1}}.$$

It remains to estimate the $(q-d)$-th root of the first factor. Put $r=q-d$. By Stirling's estimate, $(r!)^{1/r}\ge c r$ for an absolute constant $c>0$. For fixed $d$, the factor $(d!/(2q))^{1/r}$ is bounded below by a positive constant depending only on $d$, and $r=q-d\ge q/(d+1)$. Therefore
$\left(\frac{d!(q-d)!}{2q}\right)^{1/(q-d)}\ge c_dq$.
It follows that
$m\ge c_d qn/(s^{d/(q-d)}p^{(q-1)/(q-d)})$,
and a point contained in $m$ members of $\F$ has the required depth.
\end{proof}

We use the standard duality lemma from the Alon--Kleitman proof; see~\cite[Lemma~2.4]{KST18}.

\begin{lemma}[LP duality]\label{lem:lp}
Let $0<\gamma\le1$, and let $\F$ be a finite family of sets. Suppose that for every finite multiset $\F'$ of members of $\F$, some point belongs to at least $\gamma|\F'|$ members of $\F'$, counted with multiplicity. Then there is a finite multiset $P$ of points such that every $F\in\F$ contains at least $\gamma|P|$ points of $P$, counted with multiplicity.
\end{lemma}

\begin{proposition}\label{prop:basic}
Let $d\ge2$, $p\ge q\ge d+1$, and $\delta>0$. Then
\[
\HD_d^{(s)}(p,q)=O_{d,\delta}\!\left(\left(s^{q/(q-d)}p^{(q-1)/(q-d)}\right)^{\rho_d+\delta}\right).
\]
\end{proposition}

\begin{proof}
Let $\F$ satisfy the $(p,q)$-property, and let $\F'$ be any labelled multiset of its members. The multiset $\F'$ satisfies the $((p-1)(q-1)+1,q)$-property: among that many copies there are either $p$ distinct original sets or $q$ copies of one set. Multiplying all multiplicities by a sufficiently large common integer, if necessary, allows us to apply Lemma~\ref{lem:depth} without changing any of the relevant fractions. It gives a point contained in a fraction at least
$c_dq/\bigl(s^{d/(q-d)}((p-1)(q-1)+1)^{(q-1)/(q-d)}\bigr)$
of $\F'$. Since $((p-1)(q-1)+1)\le pq$, this fraction is at least
$c_dq^{-(d-1)/(q-d)}s^{-d/(q-d)}p^{-(q-1)/(q-d)}$.
As $q\ge d+1$, the factor $q^{-(d-1)/(q-d)}$ is bounded below by a positive constant depending only on $d$. Thus, after adjusting $c_d$, the resulting fraction is at least
$\gamma:=c_ds^{-d/(q-d)}p^{-(q-1)/(q-d)}$.

Lemma~\ref{lem:lp} now gives a finite multiset $P$ of points such that every $F\in\F$ contains at least $\gamma|P|$ points of $P$, counted with multiplicity. Write $F$ as the union of at most $s$ convex components. By averaging, one of these components contains at least $\gamma|P|/s$ points of $P$. Hence every $F\in\F$ contains a convex component containing at least $\eps|P|$ points, where
$\eps=\gamma/s=c_ds^{-q/(q-d)}p^{-(q-1)/(q-d)}$.

A weak $\eps$-net for $P$ with respect to convex sets meets this component, and therefore meets $F$. Rubin's weak-net theorem supplies such a net of size
$$
O_{d,\delta}(\eps^{-(\rho_d+\delta)}) = O_{d,\delta}\!\left(\left(s^{q/(q-d)}p^{(q-1)/(q-d)}\right)^{\rho_d+\delta}\right).
$$
This net is a transversal for $\F$, proving the proposition.
\end{proof}

\begin{proof}[Proof of Theorem~\ref{thm:q=d+1}]
Set $q=d+1$ in Proposition~\ref{prop:basic}.
\end{proof}

\begin{remark}\label{rem:components}
Lemmas \ref{lem:count} and \ref{lem:depth} generalize the setting of convex sets from \cite{KST18}.
If one first replaces the $n$ original sets by all $sn$ convex components, the same count carries a factor $s^q$ rather than $s^d$, and the final component step loses another factor $s$. Thus the direct route replaces $s$ by $s^2$ in the large-$q$ bound. The lexicographic certificate therefore improves the final estimates by essentially a factor $s^{\rho_d}$. For $s=1$, Theorems~\ref{thm:q=d+1} and~\ref{thm:large-q} are the bounds obtained by combining the framework of~\cite{KST18} with Rubin's weak-net estimates.
\end{remark}

\subsection{The upper bound for large \texorpdfstring{$q$}{q}}

The improvement for large $q$ uses the following elementary dichotomy from~\cite{KST18}.

\begin{lemma}\label{lem:dichotomy}
Suppose that $\F$ satisfies the $(p,q)$-property. For $p'<p$ and $q'<q$, either $\F$ satisfies the $(p',q')$-property, or there is a subfamily $S\subseteq\F$ of size $p'$ with no $q'$ intersecting members, and then $\F\setminus S$ satisfies the $(p-p',q-q'+1)$-property.
\end{lemma}

In the bootstrapping argument below, we lower $q$ in steps of $k$, and $p$ in proportional steps $k'=\lceil kp/q\rceil$, so that the ratio of the new parameters never exceeds $p/q$. Either we reach a pair with $q$ of order $k$, or the first step that fails, together with the dichotomy, leaves a family with the $(k',k+1)$-property and an exceptional subfamily that costs at most $p-q+1$ piercing points. Thus the geometric bound of Proposition~\ref{prop:basic} is applied only at the smaller scale $k$; Theorem~\ref{thm:large-q} then follows by choosing $k$ logarithmically so as to optimize the resulting bound.

\begin{proposition}[One bootstrapping step]\label{prop:bootstrap}
Let $d\ge2$, $p>q\ge d+1$, $\delta>0$, and choose an integer $k$ with $d\le k<q$. Put $t=k+1-d$. Then
\[
\HD_d^{(s)}(p,q)\le p-q+1+O_{d,\delta}\!\left(\left(s^{1+d/t}(kp/q)^{k/t}\right)^{\rho_d+\delta}\right).
\]
\end{proposition}

\begin{proof}
Put $k'=\lceil kp/q\rceil$, $p_\ell=p-\ell k'$, and $q_\ell=q-\ell k$. Choose the largest $\ell$ such that $q_\ell>k$ and $\F$ satisfies the $(p_\ell,q_\ell)$-property.

If $q_{\ell+1}\le k$, then $k<q_\ell\le2k$ and $p_\ell\le(p/q)q_\ell\le2kp/q$. Proposition~\ref{prop:basic}, applied to $(p_\ell,q_\ell)$, gives the displayed bound.

Otherwise $\F$ satisfies $(p_\ell,q_\ell)$ but not $(p_{\ell+1},q_{\ell+1})$. Applying Lemma~\ref{lem:dichotomy} inside the pair $(p_\ell,q_\ell)$ produces a subfamily $S$ of size $p_{\ell+1}$ such that $\F\setminus S$ satisfies the $(k',k+1)$-property. Proposition~\ref{prop:basic} gives the displayed bound for $\F\setminus S$.

Finally, any subfamily $S$ of size at most $p$ can be pierced by $p-q+1$ points. If $|S|\le p-q$, pierce its members separately. Otherwise extend $S$ to $p$ members. Among the guaranteed $q$ intersecting sets, at least $|S|-(p-q)$ belong to $S$; pierce these by one point and the remaining at most $p-q$ members separately.
\end{proof}

\begin{proof}[Proof of Theorem~\ref{thm:large-q}]
Write $A=p/q$ and choose $k$ of order $d+\log(e sA)$. If $q\ge C_d\log(e sp)$, then $k<q$. With $t=k+1-d$, one has $s^{1+d/t}(Ak)^{k/t}=O_d(sA\log(e sA))$. Proposition~\ref{prop:bootstrap} now gives the theorem.
\end{proof}

\begin{remark}
The almost exact bound $p-q+2$ proved in~\cite{KST18} for very large $q$ uses the sharp Hadwiger--Debrunner theorem for convex sets. For general $s$-convex sets, not only does that final step have no direct analogue: Theorem~\ref{thm:large-q-lower} gives the lower bound $p-q+d(s-1)+1$. In dimension one, Theorem~\ref{thm:one-dimensional-polylog} shows that the correct additive order is linear in $s$ and gives the explicit upper bound $p-q+2s+1$ in its stated range.
\end{remark}

\section{The one-dimensional case}\label{sec:one-dimensional}

In $\R$, an $s$-convex set is an $s$-interval, that is, a union of at most $s$ compact intervals. 
An important difference between this setting and the setting of $d \geq 2$ is that families of intervals on the line have weak $\eps$-nets of size $O(1/\eps)$. Indeed, let $x_1\le\cdots\le x_n$ be the order statistics of a finite multiset of points, and put $m=\lceil\eps n\rceil$. The set $\{x_m,x_{2m},\ldots,x_{\lfloor n/m\rfloor m}\}$ clearly meets every interval containing at least $\eps n$ points.

The table below summarizes our results for $d=1$ in the main parameter ranges and their relation to earlier results.

\subsection{A general upper bound}

As we saw in the introduction, for $q=2$, the dependence on the number of components is nearly understood. 
For $q>2$, Zerbib~\cite{Zerbib19} proved that, for fixed $p\ge q>1$, every family of $s$-intervals with the $(p,q)$-property can be pierced by $O_{p,q}(s^{q/(q-1)})$ points, where the dependence on $p$ is polynomial. Theorem~\ref{thm:one-dimensional-intro} recovers the same exponent of $s$, with explicit linear dependence on $p$. Let us recall its statement.

\medskip \noindent \textbf{Theorem~\ref{thm:one-dimensional-intro}.}
For all $p\ge q\ge2$ and $s\ge1$, one has $\HD_1^{(s)}(p,q)=O\!\left(p s^{q/(q-1)}\right)$. If $p>q$ and $q\ge C\log(e sp)$, then
\[
\HD_1^{(s)}(p,q)\le p-q+1+O\!\left(s \cdot \frac{p}{q} \cdot \log \frac{e sp}{q}\right).
\]

\begin{center}
\small
\setlength{\tabcolsep}{5pt}
\renewcommand{\arraystretch}{1.15}

\begin{tabularx}{\textwidth}{@{}P{0.17\textwidth} Y Y P{0.14\textwidth}@{}}
\toprule
Parameter range & Earlier results & Bounds in this paper & Reference \\
\midrule

$q=2$
&
\cite{Kaiser97}: every family with the $(p,2)$-property can be pierced by at most $(s^2-s+1)(p-1)$ points; \cite{Matousek01}: some separated families with the $(p,2)$-property require $\Omega\!\left(s^2(p-1)/(\log s)^2\right)$ points.
&
The general theorem recovers $H_s(p,2)=O(ps^2)$.
&
Theorem~\ref{thm:one-dimensional-intro}
\\
\addlinespace[0.65em]

Fixed $p\ge q>2$, $s\to\infty$
&
\cite{Zerbib19}: $H_s(p,q)=O_{p,q}\!\left(s^{q/(q-1)}\right)$, polynomial dependence on $p$.
&
$H_s(p,q)=O\!\left(p\,s^{q/(q-1)}\right)$, with explicit linear dependence on $p$.
&
Theorem~\ref{thm:one-dimensional-intro}
\\
\addlinespace[0.65em]

$q\ge (p-q+1)\lceil\log_2 s\rceil+1$
&
\cite{Zerbib19}: $H_s(t,t)< t^{1/(t-1)}s^{t/(t-1)}+s$ for every $t\ge2$.
&
$p-q+s\le H_s(p,q)\le p-q+1+5s$.
&
Theorem~\ref{thm:one-dimensional-near-diagonal}
\\
\addlinespace[0.65em]

$q\ge C_0 s\log(2s)\log(ep)$
&
The earlier general bounds do not give matching dependence on $s$ in this range.
&
$p-q+s\le H_s(p,q)\le p-q+2s+1$.
&
Theorem~\ref{thm:one-dimensional-polylog}; upper bound in Appendix~\ref{app:one-dimensional-bootstrapping}
\\
\addlinespace[0.65em]

Fixed $s$, $q\ge C_s\log(ep)$
&
The earlier bounds above do not yield a concentration statement.
&
$H_s(p,q)$ is one of two consecutive values.
&
Theorem~\ref{thm:one-dimensional-concentration}
\\

\bottomrule
\end{tabularx}

\medskip
\textit{Summary of the one-dimensional bounds. Here $H_s(p,q)=\HD_1^{(s)}(p,q)$.}
\end{center}

\begin{proof}
For the first assertion, we repeat the proof of Proposition~\ref{prop:basic} with $d=1$, replacing Rubin's weak-net theorem by the linear weak-net bound for intervals. Lemmas~\ref{lem:count} and~\ref{lem:depth} remain valid in this case. Let $\F$ be a family of $s$-intervals satisfying the $(p,q)$-property, and let $\F'$ be any labelled multiset of members of $\F$. Put $p'=(p-1)(q-1)+1$. Then $\F'$ satisfies the $(p',q)$-property as before. After multiplying all multiplicities by a common integer if necessary, which does not change the relevant fractions, Lemma~\ref{lem:depth}, applied with parameters $p'$ and $q$, gives a point contained in a fraction at least $cq/(s^{1/(q-1)}p')$ of $\F'$. Since $p'\le pq$, this fraction is at least $c/(p s^{1/(q-1)})$, for an absolute constant $c>0$.

By Lemma~\ref{lem:lp}, there is a finite multiset $P$ of points such that every member of $\F$ contains at least this fraction of $P$. Each member is the union of at most $s$ intervals, so one of its interval components contains at least a fraction
$c/(p s^{1+1/(q-1)})=c/(p s^{q/(q-1)})$
of $P$. Taking $\eps=c/(p s^{q/(q-1)})$, a weak $\eps$-net for intervals has size $O(1/\eps)=O(p s^{q/(q-1)})$ and meets every member of $\F$. This proves the first assertion.

For the second assertion, repeat the proof of Proposition~\ref{prop:bootstrap} with $d=1$, using the first assertion above for the smaller parameter pairs. For every integer $1\le k<q$, the same dichotomy argument gives
$\HD_1^{(s)}(p,q)\le p-q+1+O(s^{1+1/k}kp/q)$:
when the process stops, the smaller pair has first parameter $O(kp/q)$ and second parameter at least $k+1$, so its contribution is at most $O(s^{1+1/k}kp/q)$, while the exceptional subfamily, if present, can be pierced by at most $p-q+1$ points.

Choose $k=\lceil\log(e sp/q)\rceil$. If $q\ge C\log(e sp)$ and $C$ is sufficiently large, then $k<q$. Moreover, $s^{1/k}=O(1)$ and $k=O(\log(e sp/q))$. Substituting this choice in the preceding estimate gives
$\HD_1^{(s)}(p,q)\le p-q+1+O(sp\log(e sp/q)/q)$,
as required.
\end{proof}

For $q=2$, the results of~\cite{Kaiser97} and~\cite{Alon98} give the same order $O(ps^2)$ as Theorem~\ref{thm:one-dimensional-intro}. For fixed $q>2$, Zerbib's~\cite{Zerbib19} explicit general estimate has leading dependence $p^{q/(q-1)}s^{q/(q-1)}$, together with an additional term of order $p^2s$; our estimate improves the leading dependence on $p$ to linear.

\subsection{The near-diagonal range}

The paper \cite{Zerbib19} also contains  
an improved upper bound for families of
$s$-intervals satisfying a diagonal $(t,t)$-condition. 
\begin{theorem}[Zerbib]\label{thm:zerbib-tt}
If a finite family $\mathcal H$ of $s$-intervals satisfies the $(t,t)$-property for an integer $t\ge2$, then $\tau(\mathcal H)<t^{1/(t-1)}s^{t/(t-1)}+s$.
\end{theorem}

The near-diagonal upper bound is obtained by inserting this diagonal estimate into the combinatorial dichotomy. Taken together, the one-dimensional argument uses all three inputs emphasized above: the lexicographic certificate count gives the general bound, the dichotomy reduces the near-diagonal problem to a diagonal one, and Zerbib's theorem closes the argument.

\begin{theorem}\label{thm:one-dimensional-near-diagonal}
Let $p\ge q\ge2$ and $s\ge2$, put $t=\lceil q/(p-q+1)\rceil$. If $t\ge2$, then
\[
p-q+s\le \HD_1^{(s)}(p,q)\le p-q+1+t^{1/(t-1)}s^{t/(t-1)}+s.
\]
Consequently, if $q\ge(p-q+1)\lceil\log_2s\rceil+1$, then $p-q+s\le \HD_1^{(s)}(p,q)\le p-q+1+5s$.
\end{theorem}

\begin{proof}[Proof of the upper bounds in Theorem~\ref{thm:one-dimensional-near-diagonal}]
Put $r=p-q$ and $t=\lceil q/(r+1)\rceil$. For $i=0,\ldots,r$, define $p_i=p-it$ and $q_i=q-i(t-1)$. Then $p_i-q_i=r-i$, while $p_r=q_r=:u$ and $u\ge t$.

Starting from the $(p_0,q_0)$-property, suppose that for some $i<r$ the family satisfies the $(p_i,q_i)$-property. If it also satisfies the $(p_{i+1},q_{i+1})$-property, continue. Otherwise, Lemma~\ref{lem:dichotomy} gives a subfamily $S$ of size $p_{i+1}$ such that the remaining family satisfies the $(t,t)$-property, since $p_i-p_{i+1}=t$ and $q_i-q_{i+1}+1=t$. To pierce $S$, if $|S|\le r$, pierce its members separately. Otherwise, extend $S$ to a subfamily of $p$ members. Among the guaranteed $q$ members with a common point, at least $|S|-(p-q)=|S|-r$ belong to $S$; piercing these by one point and the remaining at most $r$ members of $S$ separately uses at most $r+1=p-q+1$ points. Theorem~\ref{thm:zerbib-tt} then pierces the remaining family by fewer than $t^{1/(t-1)}s^{t/(t-1)}+s$ points.

If all $r$ reductions succeed, the original family satisfies the $(u,u)$-property. Since $u\ge t$, Theorem~\ref{thm:zerbib-tt} again gives at most $t^{1/(t-1)}s^{t/(t-1)}+s$ points. This proves the upper bound.

Finally, if $q\ge(r+1)\lceil\log_2s\rceil+1$, then $t-1\ge\lceil\log_2s\rceil$. Hence $t^{1/(t-1)}\le2$ and $s^{1/(t-1)}\le2$, so the term from Theorem \ref{thm:zerbib-tt} is at most $5s$.

We postpone the proof of the lower bound to Section \ref{subsec:lb}.
\end{proof}

\paragraph{The polylogarithmic range.}
The upper bound in Theorem~\ref{thm:one-dimensional-polylog} is obtained by iterating the dichotomy until it yields an $O(s)$ additive estimate and then applying Zerbib's diagonal bound at a sequence of scales to sharpen the additive term to $2s+1$. Since the iteration and its parameter bookkeeping are more technical than the other arguments in the paper, the proof is deferred to Appendix~\ref{app:one-dimensional-bootstrapping}.

\subsection{A matching lower bound}\label{subsec:lb}

\begin{proof}[Proof of the lower bound in Theorems~\ref{thm:one-dimensional-near-diagonal} and~\ref{thm:one-dimensional-polylog}]
The idea is to take complements of very small ``holes'' in $[0,1]$. We choose the holes so that any $q$ of them are too small to cover $[0,1]$, while every set of at most $s-1$ points is contained in one of them.

For every $\mathbf{x}=(x_1,\ldots,x_{s-1})\in[0,1]^{s-1}$, choose a relatively open set $U_{\mathbf{x}}\subseteq[0,1]$ that contains all the points $x_1,\ldots,x_{s-1}$, is a union of at most $s-1$ intervals, and has total length less than $1/q$. The Cartesian powers $U_{\mathbf{x}}^{\,s-1}$ form an open cover of the compact space $[0,1]^{s-1}$. Hence there is a finite subcover $U_1^{\,s-1},\ldots,U_N^{\,s-1}$. Adding further holes of the same type if necessary, we may assume that $N\ge q$. Put $G_i=[0,1]\setminus U_i$ and $\mathcal G=\{G_1,\ldots,G_N\}$.

Each $G_i$ is a union of at most $s$ compact intervals. Any $q$ members of $\mathcal G$ intersect: the corresponding $q$ holes have total length less than $1$, and therefore cannot cover $[0,1]$. On the other hand, any set of at most $s-1$ points is contained in some $U_i$, after repeating points to obtain an $(s-1)$-tuple. Those points consequently miss $G_i$. Thus every $q$ members of $\mathcal G$ intersect, but $\mathcal G$ cannot be pierced by $s-1$ points, so $\tau(\mathcal G)\ge s$.

Finally, add $p-q$ pairwise disjoint compact intervals outside $[0,1]$, all disjoint from every member of $\mathcal G$, and let $\mathcal F$ be the resulting family. Among any $p$ members of $\mathcal F$, at most $p-q$ are added intervals, so at least $q$ belong to $\mathcal G$ and therefore have a common point. Hence $\mathcal F$ satisfies the $(p,q)$-property. Any transversal needs a separate point for each of the $p-q$ added intervals and at least $s$ further points to pierce $\mathcal G$. Therefore $\tau(\mathcal F)\ge p-q+s$.

Embedding the same construction in an affine line in $\R^d$ gives the identical lower bound in every higher dimension in the usual range $q\ge d+1$.
\end{proof}

For $q=2$, the classical upper and lower bounds leave a factor of order $(\log s)^2$. In contrast, Theorem~\ref{thm:one-dimensional-polylog} gives, to the best of our knowledge, the first broad range with matching upper and lower dependence on $s$ for $q>2$: whenever $q\ge C_0s\log(2s)\log(ep)$, the additive term lies between $s$ and $2s+1$. Theorem~\ref{thm:one-dimensional-near-diagonal} gives a parameter-sensitive estimate when $q$ is especially close to $p$.

\begin{corollary}[A fixed number of components]\label{cor:one-dimensional-fixed-s}
For every fixed $s\ge2$, there is a constant $C_s>0$ such that $q\ge C_s\log(ep)$ implies
\[
p-q+s\le \HD_1^{(s)}(p,q)\le p-q+2s+1.
\]
\end{corollary}

\begin{proof}
For fixed $s$, the factor $s\log(2s)$ in Theorem~\ref{thm:one-dimensional-polylog} is a constant.
\end{proof}

\subsection{Two-value concentration for fixed \texorpdfstring{$s$}{s}}
\label{subsec:one-dimensional-concentration}

We now fix $s\ge2$ and write $H_s(p,q)=\HD_1^{(s)}(p,q)$. The first observation is that the diagonal Hadwiger--Debrunner numbers eventually stabilize.

\begin{lemma}[Diagonal stabilization]
\label{lem:diagonal-stabilization}
There exist an integer $T_s\ge2$ and an integer $\kappa_s\in\{s,\ldots,2s\}$ such that $H_s(t,t)=\kappa_s$ for every $t\ge T_s$.
\end{lemma}

\begin{proof}
The sequence $H_s(t,t)$ is nonincreasing. Indeed, a family with the $(t+1,t+1)$-property also has the $(t,t)$-property.
The lower bound in Theorem~\ref{thm:one-dimensional-near-diagonal}, applied with $p=q=t$, gives $H_s(t,t)\ge s$. On the other hand, Theorem~\ref{thm:zerbib-tt} gives
$H_s(t,t)<t^{1/(t-1)}s^{t/(t-1)}+s$. The right-hand side tends to $2s$, so for all sufficiently large $t$ it is smaller than $2s+1$. Since $H_s(t,t)$ is an integer, it is then at most $2s$. Thus $H_s(t,t)$ is a nonincreasing sequence of integers in $\{s,\ldots,2s\}$ for all sufficiently large $t$, and hence it eventually stabilizes.
\end{proof}

The next lemma is a cleanup step. Once a constant-additive estimate is available in the logarithmic range, one application lowers its additive constant by one, until the diagonal obstruction is reached.

\begin{lemma}[One cleanup step]
\label{lem:concentration-cleanup}
Fix $s$, and let $T_s$ and $\kappa_s$ be as in Lemma~\ref{lem:diagonal-stabilization}. Suppose that $K\ge\kappa_s+2$ is an integer and that there are constants $A>0$ and $p_0$ such that
$H_s(p,q)\le p-q+K$ whenever $p\ge p_0$ and $q\ge A\log(ep)$. Then there are constants $A'>0$ and $p_0'$ such that
$H_s(p,q)\le p-q+K-1$ whenever $p\ge p_0'$ and $q\ge A'\log(ep)$.
\end{lemma}

\begin{proof}
Choose $A'$ and $p_0'$ sufficiently large that, whenever $p\ge p_0'$ and $q\ge A'\log(ep)$, one has $q\ge T_s$, $p-T_s\ge p_0$, and
$q-T_s+1\ge A\log(e(p-T_s))$. This is possible because $T_s$ is fixed.

Let $\mathcal F$ satisfy the $(p,q)$-property. If $p=q$, then $q\ge T_s$ and $\tau(\mathcal F)\le\kappa_s\le K-1$. We may therefore assume that $p>q$. Put
$p'=p-T_s$ and $q'=q-T_s+1$, so that $p'-q'=p-q-1$.

Apply Lemma~\ref{lem:dichotomy} to the pair $(p',q')$. If $\mathcal F$ satisfies the $(p',q')$-property, then the assumed bound gives
$\tau(\mathcal F)\le p'-q'+K=p-q+K-1$.

Otherwise, there is a subfamily $\mathcal S\subseteq\mathcal F$ of size $p'$ with no intersecting $q'$-tuple, and $\mathcal F\setminus\mathcal S$ satisfies the $(T_s,T_s)$-property. Hence $\tau(\mathcal F\setminus\mathcal S)\le\kappa_s$.

It remains to pierce $\mathcal S$. Choose any $T_s$ members of $\mathcal F\setminus\mathcal S$. Together with $\mathcal S$ they form a $p$-subfamily, so some $q$ of these sets intersect. At most $T_s$ of them lie outside $\mathcal S$, and therefore at least $q-T_s=q'-1$ intersecting members lie in $\mathcal S$. Piercing these by one point and every remaining member of $\mathcal S$ separately gives
$\tau(\mathcal S)\le 1+p'-(q'-1)=p-q+1$.
Consequently,
$\tau(\mathcal F)\le p-q+\kappa_s+1\le p-q+K-1$.
\end{proof}

\begin{proof}[Proof of Theorem~\ref{thm:one-dimensional-concentration}]
Corollary~\ref{cor:one-dimensional-fixed-s} gives constants $A_s,B_s>0$ such that
$H_s(p,q)\le p-q+B_s$ whenever $q\ge A_s\log(ep)$. Choose an integer $K_0\ge\max\{B_s,\kappa_s+1\}$. If $K_0>\kappa_s+1$, apply Lemma~\ref{lem:concentration-cleanup} repeatedly. The number of applications is finite and depends only on $s$. After enlarging the logarithmic constant and the lower threshold for $p$ at each step, we obtain
$H_s(p,q)\le p-q+\kappa_s+1$
for all sufficiently large $p$ and all $q\ge C_s\log(ep)$.

For the matching lower bound, assume also that $q\ge T_s$. By the definition of $\kappa_s$, there is a family $\mathcal G$ of $s$-intervals with the $(q,q)$-property and $\tau(\mathcal G)=\kappa_s$. Place $\mathcal G$ in one bounded interval and add $p-q$ compact intervals that are pairwise disjoint and disjoint from every member of $\mathcal G$. Every $p$ members of the resulting family contain at least $q$ members of $\mathcal G$, and these intersect. Thus the new family has the $(p,q)$-property. Its piercing number is $p-q+\kappa_s$, because each added interval requires its own point and $\mathcal G$ requires $\kappa_s$ further points.

Finally, enlarge $p_s$ so that
$C_s\log(e p_s)\ge T_s$. Then, whenever $p\ge p_s$ and
$q\ge C_s\log(ep)$, we automatically have $q\ge T_s$.
Thus the lower-bound construction applies throughout the stated range. We conclude that
$H_s(p,q)\in\{p-q+\kappa_s,p-q+\kappa_s+1\}$ throughout the stated range.
\end{proof}

\begin{remark}
The argument determines only $s\le\kappa_s\le2s$. The exact value of $\kappa_s$ is not known to us. In particular, it is natural to ask whether $\kappa_s=s$ for every $s$.
\end{remark}

\section{Lower bounds in dimensions \texorpdfstring{$d\ge2$}{d >= 2}}\label{sec:lower}

\subsection{A lower bound on $\HD_d(p,d+1)$ in the convex case}

We begin by improving the previously known lower bound on the Hadwiger--Debrunner numbers already for ordinary convex sets. The extension from convex sets to unions will be carried out in the next subsection.

A finite point set in $\R^d$ is in general position if no $d+1$ of its points lie on a hyperplane. For fixed integers $d,h\ge3$, let $\alpha_{d,h}(N)$ be the largest integer such that every set of $N$ points in $\R^d$ with no $d+h$ points on a hyperplane contains a subset of $\alpha_{d,h}(N)$ points in general position. Suk and Zeng~\cite[Theorem~1.2]{SukZeng26} proved the following.

\begin{theorem}[Suk--Zeng]\label{thm:suk-zeng-general-position}
Let $d,h\ge3$ be fixed integers. If $d$ is odd and $dh+2>2d+2h$, then $\alpha_{d,h}(N)\le N^{1/2+o(1)}$. If $d$ is even and $dh+2>2d+3h$, then $\alpha_{d,h}(N)\le N^{1/2+o(1)}$.
\end{theorem}

We use the following consequence of Theorem~\ref{thm:suk-zeng-general-position}.

\begin{proposition}\label{prop:suk-zeng}
For every fixed $d\ge3$ and every $\eta>0$, there are arbitrarily large $N$ and an $N$-point set $P\subset\R^d$ such that no $3d$ points of $P$ lie on a hyperplane, while every subset of at least $N^{1/2+\eta}$ points contains $d+1$ points on a hyperplane.
\end{proposition}

\begin{proof}
Apply Theorem~\ref{thm:suk-zeng-general-position} with $h=2d$. If $d$ is odd, its numerical condition becomes $2d^2+2>6d$, which holds for every odd $d\ge3$. If $d$ is even, it becomes $2d^2+2>8d$, which holds for every even $d\ge4$. Thus $\alpha_{d,2d}(N)\le N^{1/2+o(1)}$.

By the definition of $\alpha_{d,2d}(N)$, for arbitrarily large $N$ there is an $N$-point set $P\subset\R^d$ with no $d+2d=3d$ points on a hyperplane whose largest subset in general position has size at most $N^{1/2+o(1)}$. For sufficiently large $N$, this is smaller than $N^{1/2+\eta}$. Hence every subset of $P$ of size at least $N^{1/2+\eta}$ fails to be in general position, and therefore contains $d+1$ points on a hyperplane.
\end{proof}

The next elementary duality observation, that is similar to the one used in \cite{KS21}, converts such a point set into a piercing example.

\begin{lemma}[Point--hyperplane duality]\label{lem:duality-lower}
Suppose that $P\subset\R^d$ has $N$ points, no $M+1$ of them lie on a hyperplane, and every $m$-point subset contains $q$ points on a hyperplane. Then there is a family of $N$ nonempty compact convex sets with the $(m,q)$-property and piercing number at least $N/M$.
\end{lemma}

\begin{proof}
After a generic change of coordinates, apply the standard point--hyperplane duality. We obtain $N$ hyperplanes such that every $m$ of them contain $q$ through a common point, while no point lies on more than $M$ hyperplanes.

Choose a closed ball containing one point of every hyperplane and one common point for every intersecting $q$-tuple, and replace each hyperplane by its intersection with the ball. The resulting sets are nonempty, compact, and convex. They still satisfy the $(m,q)$-property, and every point belongs to at most $M$ of them. Hence every transversal has at least $N/M$ points.
\end{proof}

\begin{proof}[Proof of Theorem~\ref{thm:convex-lower}]
Fix $\eta>0$ and take the point set from Proposition~\ref{prop:suk-zeng}. Lemma~\ref{lem:duality-lower}, with $m=\lceil N^{1/2+\eta}\rceil$, $q=d+1$, and $M=3d-1$, gives a family with the $(m,d+1)$-property and piercing number at least $N/(3d-1)$.

Writing $p=m$, we have $N\ge c_{d,\eta}p^{1/(1/2+\eta)}$. Since $\eta$ can be arbitrarily small, this is at least $c_{d,\gamma}p^{2-\gamma}$ for every fixed $\gamma>0$ and arbitrarily large $p$.
\end{proof}

\subsection{Amplification by unions}
The passage from convex sets to unions is given by a simple amplification lemma.

\begin{lemma}[Amplification by unions]\label{lem:amplification}
For all $u\ge q\ge d+1$ and $s\ge1$, one has $\HD_d^{(s)}(2(u-1)(q-1)+1,q)\ge (s+1)\HD_d(u,q)$.
\end{lemma}

\begin{proof}[Proof of Lemma~\ref{lem:amplification}]
Let $\mathcal B$ be a family of compact convex sets with the $(u,q)$-property and piercing number $t=\HD_d(u,q)$. Place $2s$ affine copies $\mathcal B_1,\ldots,\mathcal B_{2s}$ in pairwise disjoint balls.

Form a new family $\F$ as follows. Choose $s$ of the balls, choose one member of the corresponding copy in each ball, and take the union of the $s$ chosen sets. Thus every member of $\F$ is $s$-convex.

Put $R=(u-1)(q-1)+1$ and take $2R-1$ members of $\F$. They use altogether $s(2R-1)>2s(R-1)$ ball occurrences, so one ball occurs in at least $R$ of the unions. Inside that ball we see a labelled multiset of $R$ members of $\mathcal B$. Either it contains $u$ distinct members, in which case $q$ intersect by the $(u,q)$-property, or one member occurs at least $q$ times. Hence the same $q$ unions intersect. Thus $\F$ has the $(2R-1,q)$-property.

Now let $N$ pierce $\F$. Call a ball covered if the points of $N$ inside it pierce the entire copy $\mathcal B_j$. If $s$ balls were not covered, we could choose in each one a member missed by $N$; their union would be an unpierced member of $\F$. Hence at least $s+1$ balls are covered. Each covered copy requires at least $t$ points, and the balls are disjoint, so $|N|\ge(s+1)t$.
\end{proof}

\begin{proof}[Proof of Theorem~\ref{thm:s-lower}]
Apply Lemma~\ref{lem:amplification} with $q=d+1$ to the families supplied by Theorem~\ref{thm:convex-lower}. If the base parameter is $u$, the new parameter is $p=2d(u-1)+1$, so $u=\Theta_d(p)$. Therefore $\HD_d^{(s)}(p,d+1)\ge(s+1)c_{d,\gamma}u^{2-\gamma}\ge c'_{d,\gamma}s p^{2-\gamma}$.
\end{proof}

\begin{remark}
Lemma~\ref{lem:amplification} transfers any convex lower bound to $s$-convex sets with a linear gain in $s$. For example, applying it to the planar bounds of~\cite{KS21} gives the lower bound $\HD_2^{(s)}(p,q)\ge s p^{1+\Omega(1/q)}$ for unions in the plane.
\end{remark}

\subsection{A lower bound on $\HD_d^{(s)}(p,q)$ for an arbitrarily large $q$}

Unlike the preceding amplification argument, which starts from a lower-bound example for ordinary convex sets, we now construct the family directly. Inside a simplex, we build almost-full $s$-convex sets that can avoid any prescribed $d(s-1)$ points. 

We use the following consequence of the Steiner convex partition theorem of Dumitrescu, Har-Peled and T\'oth~\cite{DumitrescuHarPeledToth14}. A Steiner convex partition tiles the convex hull of a finite point set by convex bodies whose interiors avoid the points, and every $n$ points in $\R^d$ admit such a partition with at most $\lceil(n-1)/d\rceil$ tiles.

\begin{lemma}[Small holes avoiding prescribed points]\label{lem:small-holes}
Let $B\subset\R^d$ be a full-dimensional simplex, let $T\subseteq B$ have at most $d(s-1)$ points, and let $\eta>0$. Then there is a compact $s$-convex set $F\subseteq B$ such that $F\cap T=\emptyset$ and $\operatorname{vol}(B\setminus F)<\eta$.
\end{lemma}

\begin{proof}
Adjoin the vertices of $B$ to $T$. The resulting set has convex hull $B$ and at most $ds+1$ points. By the Steiner convex partition theorem, $B$ can therefore be tiled by at most $\lceil(ds+1-1)/d\rceil=s$ convex bodies whose interiors avoid all these points.

Shrink each tile slightly, by the same homothety factor $1-\lambda$, about an interior point, and let $F_\lambda$ be the union of the shrunken tiles. Then $F_\lambda$ is compact, is a union of at most $s$ convex sets, and avoids $T$. Since volume is multiplied by $(1-\lambda)^d$, we have $\operatorname{vol}(B\setminus F_\lambda)=\bigl(1-(1-\lambda)^d\bigr)\operatorname{vol}(B)$. This is smaller than $\eta$ when $\lambda>0$ is sufficiently small.
\end{proof}

\begin{proof}[Proof of Theorem~\ref{thm:large-q-lower}]
For $d=1$, this is the lower-bound construction from Section~\ref{sec:one-dimensional}. Assume $d\ge2$, let $B$ be a full-dimensional simplex in $\R^d$, and choose $\eta<\operatorname{vol}(B)/q$.

For every ordered $d(s-1)$-tuple $\mathbf{x}$ of points of $B$, Lemma~\ref{lem:small-holes} gives a compact $s$-convex set $G_{\mathbf{x}}\subseteq B$ avoiding all entries of $\mathbf{x}$, whose complement $U_{\mathbf{x}}=B\setminus G_{\mathbf{x}}$ has volume less than $\eta$. Thus $\mathbf{x}\in U_{\mathbf{x}}^{\,d(s-1)}$, and these Cartesian powers form an open cover of the compact space $B^{d(s-1)}$. Choose a finite subcover $U_1^{\,d(s-1)},\ldots,U_N^{\,d(s-1)}$, adding further such holes if necessary so that $N\ge q$, and put $G_i=B\setminus U_i$.

Every $G_i$ is compact and $s$-convex. Any $q$ of them intersect, since the corresponding $q$ holes have total volume less than $q\eta<\operatorname{vol}(B)$ and therefore cannot cover $B$. On the other hand, any collection of at most $d(s-1)$ points misses some $G_i$: discard the points outside $B$, extend the remaining points to an ordered $d(s-1)$-tuple, and use the finite subcover to find a hole containing all of them. Hence $\tau(\{G_1,\ldots,G_N\})\ge d(s-1)+1$.

Finally, add $p-q$ pairwise disjoint compact convex sets outside $B$, also disjoint from every $G_i$. Among any $p$ members of the enlarged family, at least $q$ are members of $\{G_1,\ldots,G_N\}$ and therefore intersect. Thus the enlarged family has the $(p,q)$-property. The added sets require $p-q$ separate piercing points, in addition to the $d(s-1)+1$ points required inside $B$. Therefore $\HD_d^{(s)}(p,q)\ge p-q+d(s-1)+1$.
\end{proof}

\section*{Acknowledgements}

\paragraph{Use of generative AI.} The authors used ChatGPT as an auxiliary tool in discussions of computational and technical details and in editing parts of the exposition.

\printbibliography

@article{HadwigerDebrunner57,
  author  = {Hadwiger, Hugo and Debrunner, Hans},
  title   = {{\"U}ber eine Variante zum Hellyschen Satz},
  journal = {Archiv der Mathematik},
  volume  = {8},
  pages   = {309--313},
  year    = {1957},
  doi     = {10.1007/BF01898794}
}

@article{AlonKleitman92,
  author  = {Alon, Noga and Kleitman, Daniel J.},
  title   = {Piercing Convex Sets and the Hadwiger--Debrunner $(p,q)$-Problem},
  journal = {Advances in Mathematics},
  volume  = {96},
  number  = {1},
  pages   = {103--112},
  year    = {1992},
  doi     = {10.1016/0001-8708(92)90052-M}
}

@article{KatchalskiLiu79,
  author  = {Katchalski, Meir and Liu, Andrew C. F.},
  title   = {A Problem of Geometry in $\mathbb{R}^n$},
  journal = {Proceedings of the American Mathematical Society},
  volume  = {75},
  number  = {2},
  pages   = {284--288},
  year    = {1979},
  doi     = {10.1090/S0002-9939-1979-0532152-6}
}

@article{AlonBFK92,
  author  = {Alon, Noga and B{\'a}r{\'a}ny, Imre and
             F{\"u}redi, Zolt{\'a}n and Kleitman, Daniel J.},
  title   = {Point Selections and Weak $\varepsilon$-Nets for Convex Hulls},
  journal = {Combinatorics, Probability and Computing},
  volume  = {1},
  number  = {3},
  pages   = {189--200},
  year    = {1992},
  doi     = {10.1017/S0963548300000225}
}

@article{ChazelleEtAl95,
  author  = {Chazelle, Bernard and Edelsbrunner, Herbert and
             Grigni, Michelangelo and Guibas, Leonidas J. and
             Sharir, Micha and Welzl, Emo},
  title   = {Improved Bounds on Weak $\varepsilon$-Nets for Convex Sets},
  journal = {Discrete \& Computational Geometry},
  volume  = {13},
  pages   = {1--15},
  year    = {1995},
  doi     = {10.1007/BF02574025}
}

@article{MatousekWagner04,
  author  = {Matou{\v{s}}ek, Ji{\v{r}}{\'i} and Wagner, Uli},
  title   = {New Constructions of Weak $\varepsilon$-Nets},
  journal = {Discrete \& Computational Geometry},
  volume  = {32},
  number  = {2},
  pages   = {195--206},
  year    = {2004},
  doi     = {10.1007/s00454-004-1116-4}
}

@inproceedings{AlonSmorodinsky26,
  author    = {Alon, Noga and Smorodinsky, Shakhar},
  title     = {Extended {VC}-Dimension, and Radon and Tverberg Type
               Theorems for Unions of Convex Sets},
  booktitle = {Proceedings of the 2026 Annual ACM--SIAM Symposium
               on Discrete Algorithms (SODA)},
  pages     = {6223--6232},
  publisher = {SIAM},
  year      = {2026}
}

@misc{GeShuXu26,
  author        = {Ge, Gennian and Shu, Yang and Xu, Zixiang},
  title         = {Tverberg's Theorem for Unions of Convex Sets:
                   Sharp Bounds and Colored Extensions},
  year          = {2026},
  eprint        = {2607.12449},
  archivePrefix = {arXiv}
}

@misc{RocheNewton26,
  author        = {Roche-Newton, Oliver},
  title         = {A General-Position Problem for Planar Line Arrangements},
  year          = {2026},
  eprint        = {2607.25742},
  archivePrefix = {arXiv}
}

@article{BMN11,
  author  = {Bukh, Boris and Matou{\v{s}}ek, Jiri and Nivasch, Gabriel},
  title   = {Lower bounds for weak epsilon-nets and stair-convexity},
  journal = {Israel Journal of Mathematics},
  volume  = {182},
  pages   = {199--228},
  year    = {2011},
  doi     = {10.1007/s11856-011-0029-1}
}

@article{BS18,
  author  = {J. Balogh and J. Solymosi},
  title   = {On the number of points in general position in the plane},
  journal = {Discrete Analysis},
  pages   = {2018:16:1--20},
  year    = {2018}
}

@article{Kaiser97,
  author  = {Kaiser, Tom{\'a}{\v{s}}},
  title   = {Transversals of $d$-Intervals},
  journal = {Discrete \& Computational Geometry},
  volume  = {18},
  number  = {2},
  pages   = {195--203},
  year    = {1997},
  doi     = {10.1007/PL00009315}
}

@article{Alon98,
  author  = {Alon, Noga},
  title   = {Piercing $d$-Intervals},
  journal = {Discrete \& Computational Geometry},
  volume  = {19},
  number  = {3},
  pages   = {333--334},
  year    = {1998},
  doi     = {10.1007/PL00009349}
}

@article{AlonKalai95,
  author  = {Noga Alon and Gil Kalai},
  title   = {Bounding the piercing number},
  journal = {Discrete \& Computational Geometry},
  volume  = {13},
  number  = {1},
  pages   = {245--256},
  year    = {1995},
  doi     = {10.1007/BF02574042}
}

@article{Amenta96,
  author  = {Nina Amenta},
  title   = {A short proof of an interesting {Helly}-type theorem},
  journal = {Discrete \& Computational Geometry},
  volume  = {15},
  number  = {4},
  pages   = {423--427},
  year    = {1996},
  doi     = {10.1007/BF02711517}
}

@article{KalaiMeshulam08,
  author  = {Gil Kalai and Roy Meshulam},
  title   = {Leray numbers of projections and a topological {Helly}-type theorem},
  journal = {Journal of Topology},
  volume  = {1},
  number  = {3},
  pages   = {551--556},
  year    = {2008},
  doi     = {10.1112/jtopol/jtn010}
}

@article{KST18,
  author  = {Chaya Keller and Shakhar Smorodinsky and G{\'a}bor Tardos},
  title   = {Improved bounds on the {Hadwiger--Debrunner} numbers},
  journal = {Israel Journal of Mathematics},
  volume  = {225},
  number  = {2},
  pages   = {925--945},
  year    = {2018},
  doi     = {10.1007/s11856-018-1685-1}
}

@article{Larman68,
  author  = {David G. Larman},
  title   = {Helly type properties of unions of convex sets},
  journal = {Mathematika},
  volume  = {15},
  number  = {1},
  pages   = {53--59},
  year    = {1968},
  doi     = {10.1112/S0025579300002370}
}

@book{Matousek02,
  author    = {Ji{\v{r}}{\'i} Matou{\v{s}}ek},
  title     = {Lectures on Discrete Geometry},
  series    = {Graduate Texts in Mathematics},
  volume    = {212},
  publisher = {Springer},
  address   = {New York},
  year      = {2002},
  doi       = {10.1007/978-1-4613-0039-7}
}

@article{Matousek97,
  author  = {Ji{\v{r}}{\'i} Matou{\v{s}}ek},
  title   = {A {Helly}-type theorem for unions of convex sets},
  journal = {Discrete \& Computational Geometry},
  volume  = {18},
  number  = {1},
  pages   = {1--12},
  year    = {1997},
  doi     = {10.1007/PL00009305}
}

@inproceedings{Rubin21,
  author        = {Natan Rubin},
  title         = {Stronger Bounds for Weak Epsilon-Nets in Higher Dimensions},
  booktitle     = {Proceedings of the 53rd Annual ACM SIGACT Symposium on Theory of Computing},
  pages         = {989--1002},
  publisher     = {ACM},
  year          = {2021},
  doi           = {10.1145/3406325.3451062},
  eprint        = {2104.12654},
  archivePrefix = {arXiv},
  primaryClass  = {cs.CG},
  note          = {Revised arXiv version, 2023}
}

@article{Rubin22,
  author  = {Natan Rubin},
  title   = {An Improved Bound for Weak Epsilon-Nets in the Plane},
  journal = {Journal of the ACM},
  volume  = {69},
  number  = {5},
  pages   = {32:1--32:35},
  year    = {2022},
  doi     = {10.1145/3555985}
}

@article{deCaen83,
  author  = {D. de Caen},
  title   = {Extension of a theorem of {Moon} and {Moser} on complete subgraphs},
  journal = {Ars Combinatoria},
  volume  = {16},
  pages   = {5--10},
  year    = {1983}
}

@article{KS21,
  author  = {Chaya Keller and Shakhar Smorodinsky},
  title   = {A new lower bound on {Hadwiger--Debrunner} numbers in the plane},
  journal = {Israel Journal of Mathematics},
  volume  = {244},
  number  = {2},
  pages   = {649--680},
  year    = {2021},
  doi     = {10.1007/s11856-021-2185-2}
}

@article{SukZeng26,
  author  = {Andrew Suk and Ji Zeng},
  title   = {On higher dimensional point sets in general position},
  journal = {Combinatorics, Probability and Computing},
  volume  = {35},
  number  = {1},
  pages   = {134--148},
  year    = {2026},
  doi     = {10.1017/S0963548325100254},
  eprint  = {2211.15968},
  archivePrefix = {arXiv},
  primaryClass = {math.CO}
}

@inproceedings{BulavkaGoodarziTancer21,
  author    = {Denys Bulavka and Afshin Goodarzi and Martin Tancer},
  title     = {Optimal Bounds for the Colorful Fractional {Helly} Theorem},
  booktitle = {37th International Symposium on Computational Geometry (SoCG 2021)},
  series    = {Leibniz International Proceedings in Informatics (LIPIcs)},
  volume    = {189},
  pages     = {19:1--19:14},
  publisher = {Schloss Dagstuhl--Leibniz-Zentrum fuer Informatik},
  year      = {2021},
  doi       = {10.4230/LIPIcs.SoCG.2021.19},
  eprint    = {2010.15765},
  archivePrefix = {arXiv},
  primaryClass  = {math.CO}
}

@article{Zerbib19,
  author  = {Shira Zerbib},
  title   = {The $(p,q)$ Property in Families of $d$-Intervals and $d$-Trees},
  journal = {Discrete Mathematics},
  volume  = {342},
  number  = {4},
  pages   = {1089--1097},
  year    = {2019},
  doi     = {10.1016/j.disc.2018.12.019},
  eprint  = {1703.02939},
  archivePrefix = {arXiv},
  primaryClass  = {math.CO}
}

@article{Matousek01,
  author  = {Matou{\v{s}}ek, Ji{\v{r}}{\'i}},
  title   = {Lower Bounds on the Transversal Numbers of $d$-Intervals},
  journal = {Discrete \& Computational Geometry},
  volume  = {26},
  number  = {3},
  pages   = {283--287},
  year    = {2001},
  doi     = {10.1007/s00454-001-0037-8}
}

@article{DumitrescuHarPeledToth14,
  author  = {Dumitrescu, Adrian and Har-Peled, Sariel and T{\'o}th, Csaba D.},
  title   = {Minimum Convex Partitions and Maximum Empty Polytopes},
  journal = {Journal of Computational Geometry},
  volume  = {5},
  number  = {1},
  pages   = {86--103},
  year    = {2014},
  doi     = {10.20382/jocg.v5i1a5}
}

\appendix

\section{Bootstrapping for the polylogarithmic range}
\label{app:one-dimensional-bootstrapping}

We prove the upper bound in Theorem~\ref{thm:one-dimensional-polylog}. The argument has two stages. First, we iterate the $(p,q)$ dichotomy until the general one-dimensional estimate yields a bound of the form $p-q+O(s)$. We then use Zerbib's diagonal estimate at a sequence of scales to improve the additive term to $2s+1$.

Throughout the appendix, write $H_s(p,q)=\HD_1^{(s)}(p,q)$ and
\[
\Delta_s(p,q)=H_s(p,q)-(p-q).
\]
Thus $\Delta_s(p,q)$ measures only the additive cost beyond the main term $p-q$. We also put $L=\max\{2,\lceil\log_2(2s)\rceil\}$.

We use two consequences of the preceding results. First, Theorem~\ref{thm:one-dimensional-near-diagonal} gives an absolute constant $c_1>0$ such that
\begin{equation}
\label{eq:one-dimensional-near-input}
q\ge L(p-q+1)
\quad\Longrightarrow\quad
\Delta_s(p,q)\le c_1s.
\end{equation}
Thus a constant-additive bound is already available when $q$ is large compared with the gap $p-q$.

Second, the proof of Proposition~\ref{prop:bootstrap} is dimension-independent except for the estimate applied to the final smaller pair. Replacing Proposition~\ref{prop:basic} there by the first assertion of Theorem~\ref{thm:one-dimensional-intro}, we obtain, for every integer $1\le k<q$,
$H_s(p,q)\le p-q+1+O(s^{1+1/k}kp/q)$. Taking $k=L$ and using $s^{1/L}\le2$, we obtain an absolute constant $c_2>0$ such that, whenever $q>L$,
\begin{equation}
\label{eq:one-dimensional-weak-input}
\Delta_s(p,q)\le1+a\frac{p}{q},
\qquad\text{where } a=c_2sL.
\end{equation}

The following recurrence records the only dichotomy calculation needed below. The parameter $B$ is the amount removed from $q$, while $D$ is the amount by which the gap $p-q$ is reduced.

\begin{lemma}[Dichotomy recurrence for the excess]
\label{lem:one-dimensional-recurrence}
Let $p>q\ge2$, and let $B,D$ be positive integers with $D\le p-q$ and $B\le q-2$. Then
\[
\Delta_s(p,q)\le
\max\left\{
\Delta_s(p-B-D,q-B)-D,\;
1+\max_{1\le u\le D}\Delta_s(B+D,B+u)
\right\}.
\]
\end{lemma}

\begin{proof}
Apply Lemma~\ref{lem:dichotomy} to the smaller pair $(p-B-D,q-B)$. If the family satisfies this property, then
$H_s(p,q)\le H_s(p-B-D,q-B)$. Since the gap of the smaller pair is $(p-q)-D$, this gives the first term in the displayed recurrence.

Otherwise, there is a subfamily $\mathcal S$ of size $p-B-D$ containing no intersecting $(q-B)$-tuple. Suppose that the largest intersecting subfamily of $\mathcal S$ has size $q-B-u$, where $u\ge1$.

Choose any $B+D$ members outside $\mathcal S$. Together with $\mathcal S$ they form a $p$-subfamily, so some $q$ of them intersect. At least $q-(B+D)=q-B-D$ of these intersecting sets lie in $\mathcal S$, and hence $u\le D$. Since at most $q-B-u$ of them lie in $\mathcal S$, at least $B+u$ lie outside it. As the chosen $B+D$ members were arbitrary, the family outside $\mathcal S$ satisfies the $(B+D,B+u)$-property.

A largest intersecting subfamily of $\mathcal S$ can be pierced by one point, and all remaining members separately. This uses
$p-B-D-(q-B-u)+1=p-q-D+u+1$
points. The remaining family has gap $D-u$, so after subtracting the original gap $p-q$, the total excess is $1+\Delta_s(B+D,B+u)$.
\end{proof}

The advantage of working with $\Delta_s$ is now visible. In the second branch, the exceptional family costs $p-q-D+u+1$ points, while the remaining family has gap $D-u$; these two occurrences of $D-u$ cancel. Thus the recurrence transfers a bound for the excess at a smaller pair back to the original pair at a cost of only one additional point.

We next iterate the recurrence. At each step, a small part of $q$ is reserved for the weak estimate~\eqref{eq:one-dimensional-weak-input}, while the rest of $q$ is passed to the preceding induction level. The decrease in the gap is chosen to absorb the error $ap/q$ in the weak estimate.

\begin{lemma}[Polylogarithmic bootstrapping]
\label{lem:one-dimensional-polylog-bootstrap}
There are absolute constants $K,C>0$ with the following property. If $m\ge1$,
\[
q\ge Ka(L+m)
\qquad\text{and}\qquad
q^{m+1}\ge\bigl(Ka(L+m)\bigr)^m p,
\]
then $\Delta_s(p,q)\le Cs$.
\end{lemma}

\begin{proof}
We first prove an auxiliary estimate. There are absolute constants $K_1,C_1>0$ such that, for every integer $j\ge1$,
\begin{equation}
\label{eq:iterated-auxiliary}
\begin{split}
&q\ge K_1^2a(L+j),\qquad
q^{j+1}\ge K_1^j(L+j)^{j-1}a^jLp\\
&\hspace{45mm}\Longrightarrow\quad
\Delta_s(p,q)\le C_1a+j.
\end{split}
\end{equation}
The integer $j$ counts the number of times the weak estimate is recycled. We choose $K_1$ sufficiently large throughout the proof.

Let $r=p-q$. For $j=1$, put $D=\lceil4ap/q\rceil$ and $B=L(D+1)$. The assumptions give
$B\le4aLp/q+2L\le q/2$.

If $r\le D$, then $q\ge B=L(D+1)\ge L(r+1)$, so~\eqref{eq:one-dimensional-near-input} applies.

Suppose that $r>D$. Apply Lemma~\ref{lem:one-dimensional-recurrence}. Since $q-B\ge q/2>L$, the contribution of its first branch is, by~\eqref{eq:one-dimensional-weak-input}, at most
\[
\Delta_s(p-B-D,q-B)-D
\le1+\frac{a(p-B-D)}{q-B}-D
\le1.
\]
For the second branch, the pair $(B+D,B+u)$ has gap $D-u$, and
$B+u\ge L(D-u+1)$. Hence~\eqref{eq:one-dimensional-near-input} gives
$\Delta_s(B+D,B+u)\le c_1s$. The recurrence therefore gives
$\Delta_s(p,q)\le c_1s+1$, proving~\eqref{eq:iterated-auxiliary} for $j=1$.

Assume now that~\eqref{eq:iterated-auxiliary} holds for $j$, and suppose that its hypotheses hold with $j+1$ in place of $j$. Reserve a small part of $q$ by putting
\[
q_0=\left\lceil\frac{q}{8(L+j+1)}\right\rceil,
\qquad
B=q-q_0,
\qquad
D=\left\lceil\frac{4ap}{q_0}\right\rceil.
\]
The role of $D$ is to absorb the error $ap/q_0$ if the recurrence enters its first branch.

If $r\le D$, then the lower bound on $q$ implies $p/q\le2$, provided $K_1$ is sufficiently large. Hence~\eqref{eq:one-dimensional-weak-input} gives
$\Delta_s(p,q)\le1+2a$.

Suppose that $r>D$. Apply Lemma~\ref{lem:one-dimensional-recurrence}. Since $q_0>L$, the contribution of the first branch is at most
\[
\Delta_s(p-B-D,q_0)-D
\le1+\frac{a(p-B-D)}{q_0}-D
\le1.
\]

It remains to consider a second-branch pair $(B+D,B+u)$. We claim that the induction hypothesis applies to this pair. First, $B+u\ge B\ge K_1^2a(L+j)$. For the power condition, the lower bound on $q$ gives
\[
K_1^j(L+j)^{j-1}a^jLB\le\frac{q^{j+1}}{4}.
\]
Moreover, $D=O(a(L+j+1)p/q)$, and the power hypothesis at level $j+1$ gives
\[
K_1^j(L+j)^{j-1}a^jLD\le\frac{q^{j+1}}{4}.
\]
Finally, since $B$ is all but a fraction of order $1/(L+j+1)$ of $q$, we have $B^{j+1}\ge q^{j+1}/2$. Consequently,
\[
K_1^j(L+j)^{j-1}a^jL(B+D)
\le B^{j+1}
\le(B+u)^{j+1}.
\]
Thus~\eqref{eq:iterated-auxiliary} applies and gives
$\Delta_s(B+D,B+u)\le C_1a+j$. The recurrence now yields
$\Delta_s(p,q)\le C_1a+j+1$, completing the induction.

The auxiliary estimate has an error of order $a=O(sL)$. We remove the factor $L$ by one final application of the recurrence. Put
$E=\lceil C_1a+m+2\rceil$ and $B=L(E+1)$.

Choose the constant $K$ in the statement sufficiently large. Then the two hypotheses of the lemma imply $B\le q/2$, and they imply the hypotheses of~\eqref{eq:iterated-auxiliary} with $j=m$ even after $q$ is replaced by $q-B\ge q/2$.

If $r\le E$, then $q\ge B\ge L(r+1)$, and~\eqref{eq:one-dimensional-near-input} applies. Suppose that $r>E$ and apply Lemma~\ref{lem:one-dimensional-recurrence} with $D=E$. By~\eqref{eq:iterated-auxiliary}, the first branch contributes at most
$C_1a+m-E\le-2$.
For $1\le u\le E$, the pair $(B+E,B+u)$ satisfies
$B+u\ge L(E-u+1)$, and hence~\eqref{eq:one-dimensional-near-input} gives
$\Delta_s(B+E,B+u)\le c_1s$.
The recurrence therefore yields $\Delta_s(p,q)\le c_1s+1$, proving the lemma.
\end{proof}

\begin{proof}[A preliminary bound for Theorem~\ref{thm:one-dimensional-polylog}]
Put $m=\max\{1,\lceil\log(ep/s)\rceil\}$. Since $q\le p$, the hypothesis of the theorem is nonvacuous only when $s\le p$. In this range, $L+m=O(\log(ep))$.

By increasing the absolute constant in the hypothesis of the theorem, we may assume that
$q\ge eKa(L+m)$ and $q\ge s$. Since $e^m\ge ep/s$, we have
\[
\bigl(Ka(L+m)\bigr)^mp
\le\left(\frac{q}{e}\right)^mp
\le q^ms
\le q^{m+1}.
\]
Thus Lemma~\ref{lem:one-dimensional-polylog-bootstrap} applies. We have proved that there are absolute constants $A_0,B_0>0$ such that
\begin{equation}
\label{eq:preliminary-polylog-bound}
q\ge A_0s\log(2s)\log(ep)
\quad\Longrightarrow\quad
H_s(p,q)\le p-q+B_0s.
\end{equation}
\end{proof}

It remains to replace the unspecified constant $B_0$ by the explicit value $2$. For $t\ge2$, write $D_s(t)=H_s(t,t)$. Taking $B=t-1$ and $D=1$ in Lemma~\ref{lem:one-dimensional-recurrence} gives, whenever $p>q$ and $2\le t<q$,
\begin{equation}
\label{eq:one-diagonal-cleanup}
\Delta_s(p,q)\le
\max\bigl\{
\Delta_s(p-t,q-t+1)-1,\,
D_s(t)+1
\bigr\}.
\end{equation}
Thus one step either decreases the excess by one, or finishes the argument using a diagonal estimate.

\begin{lemma}[Diagonal scales]
\label{lem:diagonal-scales}
There is an absolute constant $C>0$ such that, for all $s\ge2$ and all integers $e\ge1$, the integer
$t_e=\lceil C(1+s/e)\log(2s)\rceil+2$
satisfies $D_s(t_e)\le2s+e-1$.
\end{lemma}

\begin{proof}
By Theorem~\ref{thm:zerbib-tt}, $D_s(t)<s+s(ts)^{1/(t-1)}$. For the chosen value of $t_e$, we have $\log(t_es)=O(\log(2s))$. Taking the absolute constant $C$ sufficiently large gives
\[
\frac{\log(t_es)}{t_e-1}
\le\frac14\min\left\{1,\frac{e}{s}\right\}.
\]
If $e\le s$, then $\exp(x)\le1+2x$ for $0\le x\le1$, and hence
$(t_es)^{1/(t_e-1)}<1+e/s$.
If $e\ge s$, the same conclusion follows from
$(t_es)^{1/(t_e-1)}\le\exp(1/4)<2\le1+e/s$.
Therefore $D_s(t_e)<2s+e$. Since $D_s(t_e)$ is an integer, it is at most $2s+e-1$.
\end{proof}

\begin{proof}[Completion of the proof of Theorem~\ref{thm:one-dimensional-polylog}]
Let $A_0,B_0$ be the constants in~\eqref{eq:preliminary-polylog-bound}. Put
$M=\max\{\lceil B_0s\rceil,2s+1\}$ and $E=M-(2s+1)$. Thus $E$ is the number of units that must still be removed from the preliminary additive constant.

If $E=0$, the preliminary bound already gives the result. Otherwise, choose $t_e$ as in Lemma~\ref{lem:diagonal-scales} for $1\le e\le E$. Since $E=O(s)$,
$T:=\sum_{e=1}^E(t_e-1)=O(s(\log(2s))^2)$.

After increasing the absolute constant in Theorem~\ref{thm:one-dimensional-polylog}, its hypothesis guarantees
$q-T\ge A_0s\log(2s)\log(ep)$ and $q-T>t_1$.
The scales $t_e$ are nonincreasing in $e$, so every cleanup step below is valid.

Apply~\eqref{eq:one-diagonal-cleanup} successively with
$t_1,t_2,\ldots,t_{\min\{E,p-q\}}$. Each time the first alternative occurs, the excess decreases by one and the second parameter decreases by $t_e-1$.

Suppose that the diagonal alternative occurs at the $j$-th step. The preceding $j-1$ steps have already decreased the excess by $j-1$, and Lemma~\ref{lem:diagonal-scales} gives
\[
\Delta_s(p,q)
\le D_s(t_j)+1-(j-1)
\le2s+1.
\]

If all $E$ cleanup steps use the first alternative, the remaining pair still satisfies the hypothesis of~\eqref{eq:preliminary-polylog-bound}, because its second parameter is at least $q-T$. Its excess is therefore at most $B_0s\le M$. Since the $E$ successful steps decrease the excess by $E$, the original excess is at most
$M-E=2s+1$.

The only remaining possibility is that the gap reaches zero after $p-q<E$ successful steps. The resulting diagonal parameter is at least $q-T>t_1$. Since $D_s(t)$ is nonincreasing in $t$, repeated use of the first alternative gives
$H_s(p,q)\le D_s(t_1)\le2s$.

In all cases, $\Delta_s(p,q)\le2s+1$, or equivalently
$H_s(p,q)\le p-q+2s+1$, as required.
\end{proof}

\begin{remark}[The diagonal barrier]
\label{rem:diagonal-barrier}
The extra $1$ in the bound $p-q+2s+1$ comes from the exceptional subfamily in~\eqref{eq:one-diagonal-cleanup}. Zerbib's estimate approaches $2s$ from above and therefore gives $D_s(t)\le2s$ for sufficiently large $t$, but it does not give $D_s(t)\le2s-1$. Thus an improvement below $p-q+2s+1$ by the same cleanup method would require a stronger diagonal estimate, or an argument that avoids reducing the problem to the diagonal case.
\end{remark}

\section{The colorful fractional Helly theorem}\label{app:colorful-fh}

\begin{proof}[Proof of Theorem~\ref{thm:colorful-fh}]
As in the proof of Lemma~\ref{lem:count}, we may replace the convex components by compact convex subsets while preserving every intersecting colorful tuple.

For each intersecting colorful tuple $I=(F_1,\ldots,F_{d+1})$, let $v_I$ be the lexicographically smallest point of its intersection, and choose in every $F_i$ a component containing $v_I$. By Lemma~\ref{lem:lex}, $d$ of the chosen components already determine $v_I$. We record these components; the certificate omits exactly one color.

Among the at least $\alpha n_1\cdots n_{d+1}$ intersecting tuples, some color $i$ is omitted by at least a $1/(d+1)$ fraction of the certificates. There are at most $s^d\prod_{j\ne i}n_j$ certificates omitting color $i$. Hence one such certificate occurs at least $\alpha n_i/((d+1)s^d)$ times. It determines one point $v$, and every tuple carrying it uses a different set of color $i$ that contains $v$. Thus $v$ belongs to at least the asserted number of members of $\F_i$.
\end{proof}

\begin{remark}
If $\beta_d^{\mathrm{col}}$ is any colorful fractional Helly function for convex sets, then the direct component reduction and Theorem~\ref{thm:colorful-fh} give $\beta_{d,s}^{\mathrm{col}}(\alpha)\ge\max\{\beta_d^{\mathrm{col}}(\alpha/s^{d+1}),\alpha/((d+1)s^d)\}$.
\end{remark}





\end{document}